\documentclass[11pt]{amsart}

\usepackage{a4}
\usepackage{amsmath,amsfonts}
\usepackage{amssymb}
\usepackage{amscd}
\usepackage{amsthm}
\usepackage{comment}
\usepackage{marginnote}
\usepackage{hyperref}
\usepackage{todonotes}
\usepackage{graphicx}

\numberwithin{equation}{section}
\newtheorem{theorem}[equation]{Theorem}
\newtheorem{proposition}[equation]{Proposition}
\newtheorem{lemma}[equation]{Lemma}
\newtheorem{corollary}[equation]{Corollary}

\theoremstyle{definition}
\newtheorem{definition}[equation]{Definition} 
\newtheorem{remark}[equation]{Remark}
\newtheorem{example}[equation]{Example}

\newcommand{\R}{\mathbb R}
\newcommand{\Q}{\mathbb Q}
\newcommand{\Z}{\mathbb Z}
\newcommand{\bb}{\mathcal B}
\newcommand{\f}{\mathcal F}
\newcommand{\g}{\mathfrak{g}}
\newcommand{\HH}{\mathbb H}
\newcommand{\h}{\mathfrak{h}}

\newcommand{\card}{\operatorname{Card}}
\newcommand{\diam}{\operatorname{diam}}
\newcommand{\Span}{\operatorname{span}}

\newcommand{\measurerestr}{%
  \,\raisebox{-.127ex}{\reflectbox{\rotatebox[origin=br]{-90}{$\lnot$}}}\,%
}

\begin{document}

\title{Differentiation of measures in metric spaces}

\author{S\'everine Rigot}

\address[Rigot]{Universit\'e C\^ote d'Azur, CNRS, LJAD, France}

\email{rigot@unice.fr}

\thanks{Lecture Notes of a course given at the CIME-CIRM Course on New Trends on Analysis and Geometry in Metric Spaces, June 26-30, 2017, Levico Terme, Italy. The author is partially supported by ANR grant ANR-15-CE40-0018.}

\begin{abstract} 
The theory of differentiation of measures originates from works of Besicovitch in the 1940's. His pioneering works, as well as subsequent developments of the theory, rely as fundamental tools on suitable covering properties. The first aim of these notes is to recall nowadays classical results about differentiation of measures in the metric setting together with the covering properties on which they are based. We will then focus on one of these covering properties, called in the present notes the weak Besicovitch covering property, which plays a central role in the characterization of (complete separable) metric spaces where the differentiation theorem holds for every (locally finite Borel regular) measure. We review in the last part of these notes recent results about the validity or non validity of this covering property. 
\end{abstract}

\maketitle

\section{Introduction}

The theory of differentiation of measures originates from works of Besicovitch in the 1940's. His pioneering works, as well as subsequent developments of the theory, rely as fundamental tools on suitable covering properties. Our aim in these notes is twofold. In the first sections, we explain classical results about differentiation of measures in metric spaces together with the covering properties on which they are based. Such covering results are more generally useful tools that can be used to deduce global properties of a measure, and, for some of them, of the ambient metric space, from local ones. Our presentation, which does not aim at exhaustivity on the subject, in mainly based on~\cite{federer},~\cite{heinonen},~\cite{mattila},~\cite{P}. In the last section, we focus on one of these covering properties, namely, the weak Besicovitch covering property in the terminology of the present notes, which plays a central role in the characterization of (complete separable) metric spaces where the differentiation theorem holds for every (locally finite Borel regular) measure. We review in the last part of these notes recent results about the validity or non validity of this covering property that are mainly taken from~\cite{LeDonneRigot1},~\cite{LeDonneRigot2},~\cite{LeDonneRigot3}.

\medskip

Throughout these notes, a measure means a nonnegative, monotonic, countably subadditive set function defined on all subsets of a topological space $X$, vanishing for the empty set. All measures are furthermore assumed to be Borel regular, which means that open sets are measurable and every set is contained in a Borel set with the same measure. We say that a measure is locally finite if every point has a neighborhood with finite measure. We refer for instance to~\cite[Chapter~1]{mattila} for detailed definitions.

\medskip

We denote by $(X,d)$ a space $X$ metrized by a distance $d$. Balls in $(X,d)$ are assumed to be closed. Namely, a ball denotes a set of the form $B(x,r) := \{y\in X : d(x,y) \leq r\}$ for some $x\in X$ and some $0<r<\infty$.  

\medskip
We say that the differentiation theorem holds for a measure $\lambda$ over $(X,d)$ if
\begin{equation} \label{e:lebesgue-diff}
\lim_{r\downarrow 0} \frac{1}{\lambda(B(x,r))} \int_{B(x,r)} f(y) \, d\lambda y = f(x) \;\; \text{ for } \lambda \text{-a.e. } x\in X
\end{equation}
for every $f\in L^1_{loc} (\lambda)$, where the latter means that every point has a neighborhood where $f$ is $\lambda$-integrable.

\medskip

In Section~\ref{sect:Vitali-type-meas}, we consider a classical Vitali type covering property for a measure $\lambda$ that implies the validity of the differentiation theorem for $\lambda$ under mild additional assumptions on $(X,d)$ and/or $\lambda$. In Section~\ref{sect:doubling-meas}, we prove that doubling measures are of Vitali type, and hence form a large class of measures for which the results of Section~\ref{sect:Vitali-type-meas} apply. Section~\ref{sect:Eulidean-space} is devoted to Radon measures over the Euclidean space. We recall in this section that the Euclidean distance satisfies a strong Besicovitch covering property which implies that every Radon measure over the Euclidean space is of Vitali type. Motivated by the example of the Euclidean space, we slightly change our point of view in Section~\ref{sect:finite-dim-metrics}. We investigate covering properties for metrics over an ambient separable space that imply the validity of the differentiation theorem for every locally finite Borel regular measure. We end this section with a characterization of complete separable metric spaces for which the differentiation theorem holds for every locally finite Borel regular measure. Section~\ref{sect:wbcp} is devoted to a closer study of the weak Besicovitch covering property introduced in the previous sections. We give in this last section recent results about the validity or non validity of this covering property.

\section{Vitali type measures} \label{sect:Vitali-type-meas}

We present in this section a classical Vitali type covering property for a measure $\lambda$ over a metric space $(X,d)$ that implies the validity of the differentiation theorem for $\lambda$ under mild additional assumptions on $(X,d)$ and/or $\lambda$. Our presentation follows closely~\cite[2.9]{federer} although our setting is slightly different. We shall work here with locally finite Borel regular measures over separable metric spaces. Results in this section hold also in slighlty different contexts with minor modifications of the arguments presented here, see Remark~\ref{rk:federer-context}.

\begin{definition} [Vitali type measures] \label{def:vitali-type-meas}
A Borel regular measure $\lambda$ over a metric space $(X,d)$ is said to be of \textit{Vitali type with respect to $d$} if for every $A\subset X$ and every family $\bb$ of balls in $(X,d)$ such that each point of $A$ is the center of arbitrarily small balls of $\bb$, that is, $$\inf \{ r>0 : B(x,r) \in \bb \} = 0 \;\;\text{for } x\in A,$$ there exists a countable disjointed subfamily $\f \subset \bb$ such that the balls in $\f$ cover $\lambda$-almost all of $A$, namely, $$\lambda \left(A \setminus \bigcup_{B\in\f} B\right) = 0.$$
\end{definition}

\medskip

We denote by $M$ the class of all locally finite Borel regular measures over $(X,d)$. Recall that if $(X,d)$ is separable, every measure $\lambda \in M$ has the following useful properties. First, there exists a sequence $U_1,U_2,\cdots$ of open sets such that $X=\cup_{j\geq 1} U_j$ and $\lambda(U_j) <\infty$ for every $j\geq 1$. In particular $\lambda$ is $\sigma$-finite. Second, $\lambda$ is outer regular, which means that for every $\lambda$-measurable set $A$ and every $\epsilon>0$, there is an open set $U$ such that $A\subset U$ and $\lambda (U\setminus A) <\epsilon$.

\medskip

For a given measure $\lambda \in M$, we associate to $\mu\in M$ the following set function,
$$\mu_\lambda (A) := \inf\{ \mu(C) : C \text{ is a Borel set and } \lambda (A\setminus C) = 0\}$$ for $A\subset X$. Evidently $\mu_\lambda \leq \mu$ as set functions. Elementary facts and the relationship between $\mu$ and $\mu_\lambda$ are given in Theorem~\ref{thm:decomposition-meas} below. We recall that, given measures $\lambda$ and $\mu$ over a space $X$, we say that $\mu$ is absolutely continuous with respect to $\lambda$, and we write $\mu \ll \lambda$, if $\lambda(A)$ implies $\mu(A) = 0$ for $A\subset X$. We say that $\mu$ and $\lambda$ are mutually singular, and we write $\mu \perp \lambda$,  if there exists a set $A\subset X$ such that $\mu(A) = 0 = \lambda (X\setminus A)$.

\begin{theorem} \label{thm:decomposition-meas}
Assume that $(X,d)$ is separable. Let $\lambda \in M$. Then, for each measure $\mu\in M$, there exists a borel set $A$ such that $$\mu_\lambda = \mu \measurerestr A \;\;\text{and}\;\; \lambda (X\setminus A) =0.$$
Therefore $\mu_\lambda \in M$, $\mu_\lambda \ll \lambda$, and $$\mu = \mu_\lambda + \mu_s,$$ where 
$\mu_s := \mu \measurerestr (X\setminus A)$ and $\lambda$ are mutually singular. Furthermore $\mu \ll \lambda$ iff $\mu = \mu_\lambda$.
\end{theorem}

\begin{proof}
For a proof we refer to~\cite[2.9.2]{federer} and Remark~\ref{rk:federer-context}.
\end{proof}

\medskip
We also recall that the pair $(\mu_\lambda,\mu_s)$ given by Theorem~\ref{thm:decomposition-meas} is the unique pair of measures in $M$ such that $\mu = \mu_\lambda + \mu_s$ with $\mu_\lambda \ll \lambda$ and $\mu_s \perp \lambda$. It is called the Lebesgue decomposition of $\mu$ relative to $\lambda$.

\medskip

Let $\lambda$, $\mu\in M$. The upper and lower derivatives of $\mu$ with respect to $\lambda$ at a point $x\in X$ are defined by
\begin{gather*}
\overline{D}(\mu,\lambda,x): = \limsup_{r \rightarrow 0} \frac{\mu(B(x,r))}{\lambda(B(x,r))}, \\
\underline{D} (\mu,\lambda,x) : = \liminf_{r \rightarrow 0} \frac{\mu(B(x,r))}{\lambda(B(x,r))}.
\end{gather*}
At points $x$ where the limit exists we define the derivative of $\mu$ with respect to $\lambda$ by $$D(\mu,\lambda,x) := \overline{D}(\mu,\lambda,x) = \underline{D} (\mu,\lambda,x).$$
Here we interpret $0/0 = 0$.

\medskip

The main result in this section is given in the following theorem.

\begin{theorem} \label{thm:diff-vitali-meas}
Assume that $(X,d)$ is separable. Let $\lambda\in M$ be of Vitali type with respect to $d$. Then for every $\mu \in M$, 
\begin{align}
&0 \leq D(\mu,\lambda,x) < \infty \;\; \text{for } \lambda\text{-a.e. } x\in X, \label{e:diff-vitali-meas_1}\\
&D(\mu,\lambda,\cdot) \text{ is } \lambda \text{-measurable},\label{e:diff-vitali-meas_2} \\
&\mu_\lambda (A) = \int_A D(\mu,\lambda,x)\,d\lambda x \;\; \text{for every } \lambda \text{-measurable set } A.\label{e:diff-vitali-meas_3}
\end{align}
\end{theorem}

Before proving Theorem~\ref{thm:diff-vitali-meas}, we start with a useful lemma that has its own interest. It allows to deduce for two given measures $\alpha$, $\beta \in M$ a global comparison between $\alpha_\lambda (A)$ and $\beta_\lambda (A)$ if some pointwise estimate of $\underline{D} (\alpha,\beta,x)$ is available on $A$. Apart from Lemma~\ref{lem:upper-bound-lower-density-vs-measure}, the assumption that $\lambda$ is of Vitali type with respect to $d$ will not be used elsewhere in the proof of Theorem~\ref{thm:diff-vitali-meas}.

\begin{lemma} \label{lem:upper-bound-lower-density-vs-measure} Assume that $(X,d)$ is separable. Let $\lambda\in M$ be of Vitali type with respect to $d$. Let $\alpha, \beta \in M$ and $0<c<\infty$. Let $A\subset X$ be such that $\underline{D} (\alpha,\beta,x) <c$ for every $x\in A$. Then $\alpha_\lambda (A) \leq c\beta_\lambda (A)$.
\end{lemma} 

\begin{proof}
For each $\epsilon >0$, one can apply the definition of $\beta_\lambda$ in conjunction with the regularity of $\beta$ to find an open set $W$ such that $\lambda (A\setminus W)=0$ and $\beta (W) \leq  \beta_\lambda (A) + \epsilon$. Since $\lambda$ is of Vitali type, one can find countably many disjointed balls $B_i \subset W$ such that $\lambda((A\cap W) \setminus \cup_i B_i) = 0$ and $\alpha(B_i) <c \beta(B_i)$ for each $i$. Since $\alpha_\lambda \ll \lambda$, we get
\begin{multline*}
\alpha_\lambda (A)  \leq \alpha_\lambda \left(\bigcup_i B_i \right) \leq \alpha \left(\bigcup_i B_i \right)\\
 \leq \sum_i \alpha (B_i) \leq c \sum_i \beta (B_i) = c \beta \left(\bigcup_i B_i \right) \leq c \beta (W) \leq c(\beta_\lambda (A) +\epsilon).
\end{multline*}
The conclusion follows since this holds for every $\epsilon>0$.
\end{proof}

\begin{corollary} \label{cor:bounds-lower-upper-density-vs-meas}
Assume that $(X,d)$ is separable. Let $\lambda\in M$ be of Vitali type with respect to $d$. Let $\mu \in M$ and $0<c<\infty$. Then
\begin{equation} \label{e:bound-lower-density-vs-meas}
\underline{D}(\mu,\lambda,x) <c \text{ for every } x\in A \Rightarrow \mu_\lambda (A) \leq c \lambda (A)
\end{equation}
\begin{equation} \label{e:bound-upper-density-vs-meas}
\overline{D}(\mu,\lambda,x) > c \text{ for every } x\in A \Rightarrow \mu_\lambda (A) \geq c \lambda (A)
\end{equation}
\end{corollary}

\begin{proof}
Since $\lambda_\lambda = \lambda$,~\eqref{e:bound-lower-density-vs-meas} follows from Lemma~\ref{lem:upper-bound-lower-density-vs-measure}. For~\eqref{e:bound-upper-density-vs-meas}, we have $\overline{D}(\mu,\lambda,x) > c \Rightarrow \underline{D}(\lambda,\mu,x) <1/c$ and it follows from Lemma~\ref{lem:upper-bound-lower-density-vs-measure} that $\lambda(A) =\lambda_\lambda(A) \leq \mu_\lambda(A) / c$.
\end{proof}

\medskip

\begin{proof} [Proof of Theorem~\ref{thm:diff-vitali-meas}] To prove~\eqref{e:diff-vitali-meas_1} let us first consider $A\subset X$ such that $\lambda(A)<\infty$ and $\mu(A)<\infty$. Let $0<s<t<\infty$ and set
$$A_{s,t} := \{x\in A : \underline{D}(\mu,\lambda,x) <s<t<\overline{D}(\mu,\lambda,x)\}.$$
By Corollary~\ref{cor:bounds-lower-upper-density-vs-meas}, $$t\lambda (A_{s,t}) \leq \mu_\lambda(A_{s,t}) \leq s \lambda (A_{s,t}) < \infty.$$
Since $s<t$, this implies $\lambda(A_{s,t})= 0$ and hence $$\lambda \left(\bigcup_{\substack{0<s<t<\infty \\ s,t\in\Q}} A_{s,t} \right) = 0.$$
Let $n\geq 1$ be a positive integer and set $$A_n := \{x\in A : \overline{D}(\mu,\lambda,x) > n\} \, \text{ and } \, A_\infty := \{x\in A : \overline{D}(\mu,\lambda,x) = \infty\}.$$ By~\eqref{e:bound-upper-density-vs-meas}, we have $n\lambda(A_n) \leq \mu_\lambda(A_n)\leq \mu(A) < \infty$ hence $$\lambda(A_\infty) = \lim_{n\rightarrow +\infty} \lambda(A_n) = 0.$$ Then~\eqref{e:diff-vitali-meas_1} follows noting that $$A \setminus \{x\in A : 0\leq D(\mu,\lambda,x)  < \infty\} \subset A_\infty \bigcup \left(\bigcup_{\substack{0<s<t<\infty \\ s,t\in\Q}} A_{s,t}\right)$$ and recalling that $X$ can be written as a countable union of sets with finite $\lambda$ and $\mu$ measure. 

\medskip

For a proof of~\eqref{e:diff-vitali-meas_2}, we refer to~\cite[2.9.6]{federer} and Remark~\ref{rk:federer-context}. 

\medskip

To prove~\eqref{e:diff-vitali-meas_3},  let $A$ be $\lambda$-measurable. Then $A$ is $\mu_\lambda$-measurable, see~\cite[2.9.7]{federer} and Remark~\ref{rk:federer-context}. Let $1<t<\infty$. For $p\in\Z$, set $$A_p := \{x\in A : t^p \leq D(\mu,\lambda,x) < t^{p+1}\}.$$
By~\eqref{e:diff-vitali-meas_1}, we have 
\begin{equation*}
\int_A D(\mu,\lambda,x)\, d\lambda x = \sum_p \int_{A_p} D(\mu,\lambda,x)\, d\lambda x \leq \sum_p t^{p+1} \lambda(A_p)
\end{equation*}
and $t^p \lambda(A_p) \leq \mu_\lambda (A_p)$ by~\eqref{e:bound-upper-density-vs-meas}. Hence \begin{equation} \label{e:upper-bound}
\int_A D(\mu,\lambda,x)\, d\lambda x \leq t \sum_p \mu_\lambda (A_p) \leq t \mu_\lambda(A).
\end{equation}
On the other hand, since $\mu_\lambda \ll \lambda$,~\eqref{e:diff-vitali-meas_1} and~\eqref{e:bound-lower-density-vs-meas} imply that $\mu_\lambda(\{x\in A : D(\mu,\lambda,x) = 0\}) = 0$. Hence it follows from~\eqref{e:bound-lower-density-vs-meas} that
\begin{equation} \label{e:lower-bound}
\begin{split}
\mu_\lambda (A ) &= \sum_p \mu_\lambda (A_p) \leq t^{p+1} \sum_p \lambda (A_p) \\
&\leq t \sum_p \int_{A_p} D(\mu,\lambda,x)\, d\lambda x = t \int_A D(\mu,\lambda,x)\, d\lambda x.
\end{split}
\end{equation}
Then~\eqref{e:diff-vitali-meas_3} follows from~\eqref{e:upper-bound} and~\eqref{e:lower-bound} since these hold for every $1<t<\infty$. 
\end{proof}

\bigskip

We stress that, in addition to the classical Radon-Nikodym theorem, Theorem~\ref{thm:diff-vitali-meas} gives a concrete representation of the density of the absolutely continuous part of the Lebesgue decomposition of $\mu$ relative to $\lambda$ in terms of the derivative $D(\mu,\lambda,\cdot)$. 

\medskip

Furthermore, the validity of the differentation theorem for Vitali type measures comes as a direct consequence of Theorem~\ref{thm:diff-vitali-meas}.

\begin{corollary} \label{cor:diff-thm-from-vitali-type}
Assume that $(X,d)$ is separable. Let $\lambda\in M$ be of Vitali type with respect to $d$. Then the differentiation theorem holds for $\lambda$.
\end{corollary}

\begin{proof}
It is sufficient to prove that~\eqref{e:lebesgue-diff} holds for nonnegative $f\in L^1_{loc} (\lambda)$. For such an $f$, we define $\mu \in M$ by $$\mu(A): = \int_A^* f(x)\, d\lambda x$$ for $A\subset X$.
 We infer from Theorem~\ref{thm:diff-vitali-meas} that 
$$\int_A f(x)\, d\lambda x = \mu (A) = \mu_\lambda (A) = \int_A D(\mu,\lambda,x) \, d\lambda x$$ for every $\lambda$-measurable set A. Hence $$f(x) = D(\mu,\lambda,x) \,\text{ for }\lambda \text{-a.e. }x\in X$$
and it follows that $$f(x) = \lim_{r\downarrow 0} \frac{1}{\lambda(B(x,r))} \int_{B(x,r)} f(y) \, d\lambda y \,\text{ for }\lambda \text{-a.e. }x\in X.$$
\end{proof}

 \begin{remark} \label{rk:federer-context}
 In~\cite[2.9]{federer}, $(X,d)$ is not assumed to be separable and measures under consideration are Borel regular measures $\lambda$ such that every bounded set has 
finite $\lambda$ measure. Namely, the main results in~\cite[2.9]{federer} are gathered in the following theorem.
 
\begin{theorem} \label{thm:diff-thm-from-vitali-fed}
 Let $(X,d)$ be a metric space. Let $\lambda$, $\mu$ be Borel regular measures over $X$ such that every bounded set has finite $\lambda$ and $\mu$ measure. Assume that $\lambda$ is of Vitali type with respect to $d$. Then the conclusions of Theorem~\ref{thm:decomposition-meas} and Theorem~\ref{thm:diff-vitali-meas} hold true. Furthermore, if $f$ is an $\overline\R$-valued $\lambda$-measurable function such that $$\int_A |f(x)|\, d\lambda x <\infty$$ for every bounded $\lambda$-measurable set $A$, then 
\begin{equation*}
\lim_{r\downarrow 0} \frac{1}{\lambda(B(x,r))} \int_{B(x,r)} f(y) \, d\lambda y = f(x) \;\; \text{ for } \lambda \text{-a.e. } x\in X.
\end{equation*}
\end{theorem} 

 It is not difficult to see that minor modifications of the arguments in~\cite[2.9]{federer} can be used to handle slightly different settings. For locally finite Borel regular measures over separable metric spaces, the minor modifications we used in the proof of Theorem~\ref{thm:diff-vitali-meas} are based on the fact that locally finite Borel regular measures over separable metric spaces are $\sigma$-finite and outer regular. Similarly one can also prove that Theorem~\ref{thm:decomposition-meas}, Theorem~\ref{thm:diff-vitali-meas} and Corollary~\ref{cor:diff-thm-from-vitali-type} hold true when considering Radon measures over locally compact separable metric spaces.
 \end{remark}

\section{Doubling measures} \label{sect:doubling-meas}

An important class of Vitali type measures, see Definition~\ref{def:vitali-type-meas}, is given by doubling measures. A classical proof of this result relies on a general covering theorem often called basic $5r$ covering theorem in the literature, see Theorem~\ref{thm:basic-5r-covering}. Our presentation in the present section follows closely~\cite[Chapter~1]{heinonen}.

\medskip

Let $(X,d)$ be a metric space. For a ball $B$ with center $x$ and radius $0<r<\infty$ in $(X,d)$ and $\tau >0$, we denote by $\tau B:= B(x,\tau r)$ the concentric ball with radius $\tau r$.

\medskip

A Borel regular measure $\lambda$ over $X$ is said to be \textit{doubling with respect to $d$} if there is a constant $C\geq 1$ such that $$\lambda(2B) \leq C\lambda (B)\;\; \text{for every ball } B,$$ and $\lambda$ is nondegenerate in the sense that $\lambda(B_1) >0$ and $\lambda(B_2) < \infty$ for some balls $B_1$ and $B_2$. Note that this implies in particular that $0<\lambda(B) <\infty$ for every ball $B$.

\begin{theorem} \label{thm:Vitali-doubling-meas}
Doubling measures over a metric space $(X,d)$ are of Vitali type with respect to $d$.
\end{theorem}

A classical proof of Theorem~\ref{thm:Vitali-doubling-meas} relies on the following general covering theorem.

\begin{theorem} \label{thm:basic-5r-covering}
Every family $\bb$ of balls with uniformly bounded diameter, that is, such that $\sup\{ \diam B : B \in \mathcal{B}\} < \infty$,  in a metric space contains a disjointed subfamily $\f \subset \bb$ such that 
$$\bigcup_{B\in \bb} B \subset \bigcup_{B\in \f} 5B$$
and every ball $B$ in $\bb$ meets a ball in $\f$ with radius at least half that of $B$.
\end{theorem}

\begin{proof}
The proof relies on Zorn's lemma, see e.g.~\cite[Theorem 1.2]{heinonen}. Note that $\f$ is not asserted to be countable but in applications it often will be.  Constructive proofs are possible under mild additional assumptions on $X$, see for instance~\cite[Theorem 2.1]{mattila} where closed balls are assumed to be compact. 
\end{proof}

\noindent\textit{Proof of Theorem~\ref{thm:Vitali-doubling-meas}.} We assume first that $A$ is bounded. Let $\bb$ be a family of balls such that $\inf \{ r>0 : B(x,r) \in \bb \} = 0$ for $x\in A$. We may assume with no loss of generality that the balls in $\bb$ are centered on $A$ and have uniformly bounded diameter. By Theorem~\ref{thm:basic-5r-covering} we can find a disjointed subfamily $\f \subset \bb$ such that the balls $5B$, $B\in\f$, cover $A$. 
Since the union $\cup_\f B$ is contained in some fixed ball, and balls have positive and finite $\lambda$ measure, the family $\f = \{B_1,B_2,\cdots\}$ is necessarily countable. We infer from the doubling property of $\lambda$ that 
$$\sum_{i\geq 1} \lambda(5B_i) \leq C \sum_{i\geq 1} \lambda(B_i) = C \lambda\left(\bigcup_{i\geq 1} B_i\right) < \infty$$
for some constant $C\geq 1$, hence $$\sum_{i>N} \lambda(5B_i) \rightarrow 0$$ as $N\rightarrow \infty$. Therefore it is sufficient to show that $$A\setminus \bigcup_{i=1}^N B_i \subset \bigcup_{j>N} 5B_j.$$
Let $a\in A \setminus \cup_{i\leq N} B_i$. Since balls in $\bb$ are closed, we can find a ball $B(a,r) \in \bb$ that does not meet $\cup_{i\leq N} B_i$. On the other hand, by Theorem~\ref{thm:basic-5r-covering}, one can choose the family $\f$ in such a way that $B(a,r)$ meets some ball $B_j\in\f$ with radius at least $r/2$. Thus $j>N$ and $B(a,r) \subset 5B_j$ as required.

\medskip

To get rid of the assumption that $A$ is bounded, we fix some $x\in X$ and an increasing sequence of positive numbers $R_0,R_1,\cdots \rightarrow\infty$ so that $\lambda (\{y\in X : d(x,y)=R_k\}) = 0$ for every $k\geq 0$. This is possible since $\lambda (\{y\in X : d(x,y)=R\}) = 0$ for a.e.~$R>0$. We set $U_0 := \{y\in X : d(x,y)<R_0\}$ and $U_k := \{y\in X : R_{k-1} < d(y,x) < R_k\}$ for $k\geq 1$. We apply the above construction to the sets $A\cap U_k$ for $k\geq 0$  to find disjointed countable families of balls $\f_k$ such that the balls in $\f_k$ are contained in $U_k$ and cover $\lambda$-almost all of $A\cap U_k$. Then $\f = \cup_k \f_k$ gives the conclusion. \hfill $\Box$

\begin{remark} \label{rk:asymptotic-doubling-vitali} Theorem~\ref{thm:Vitali-doubling-meas} holds more generally for asymptotically doubling measures, that is, when the doubling condition is relaxed into $$\limsup_{r\rightarrow 0} \frac{\lambda(B(x,2r))}{\lambda(B(x,r))} <\infty \;\; \text{ for } \lambda \text{-a.e. } x\in X,$$ 
see~\cite[2.8.17]{federer}.
\end{remark}

\section{Radon measures in Euclidean spaces} \label{sect:Eulidean-space}

The theory of differentiation of measures originates from works of Besicovitch in the Euclidean space (\cite{B1},~\cite{B2}). Besicovitch proved that every locally finite Borel regular measure, or equivalently Radon measure, over the Euclidean space is of Vitali type and hence satisfies the differentiation theorem. His proof relies on a geometric property  of the Euclidean distance given in Lemma~\ref{lem:WBCP-Rn} below and called weak Besicovitch covering property (WBCP) in the present notes. Lemma~\ref{lem:WBCP-Rn} turns out to imply two stronger covering properties for the Euclidean distance, namely, Theorem~\ref{thm:strong-bcp-Rn}~$(i)$ and Theorem~\ref{thm:strong-bcp-Rn}~$(ii)$, called respectively Besicovitch covering property (BCP) and strong Besicovitch covering property in our terminology. The fact that every Radon measure over $\R^n$ is of Vitali type with respect to the Euclidean distance can then be obtained as a rather simple consequence of Theorem~\ref{thm:strong-bcp-Rn}~$(ii)$, see Theorem~\ref{thm:Vitali-Radon-meas-Rn}. BCP and WBCP  will be studied in more details in the general metric setting in Sections~\ref{sect:finite-dim-metrics} and~\ref{sect:wbcp}. Our presentation in the present section follows~\cite[Chapter 2]{mattila}.

\medskip

We start with a lemma which is quite elementary in the Euclidean setting. However, as we shall see in Section~\ref{sect:finite-dim-metrics}, it turns out to play a central role for the theory of differentiation of arbitrary measures when rephrased in  the metric setting.

\begin{lemma} [weak Besicovitch covering property for the Euclidean distance] \label{lem:WBCP-Rn} Let $\R^n$ be equipped with the Euclidean distance. There exists an integer $K\geq 1$ with the following property. Assume that there exist $k$ points $x_1,\dots,x_k$ in $\R^n$ and $k$ positive numbers $r_1,\dots,r_k$ such that $$x_i\not \in B(x_j,r_j) \;\; \text{for } j\not= i, \; \text{and } \bigcap_{i=1}^k B(x_i,r_i) \not= \emptyset.$$ Then $k\leq K$.
\end{lemma}

\begin{proof}
Without loss of generality, we may assume that $x_i\not = 0$ for $i=1,\dots,k$, and $0\in \cap_{i=1}^k B(x_i,r_i)$. Then $\|x_i\| \leq r_i < \|x_i - x_j\|$ for $i\not= j$ where $\|\cdot\|$ denotes the Euclidean norm. It follows from elementary geometric arguments that the angle between $x_i$ and $x_j$ for $i\not= j$ is at least $60^o$, that is, $$\|\frac{x_i}{\|x_i\|} - \frac{x_j}{\|x_j\|} \| \geq 1 \;\text{ for } i\not= j,$$ see~\cite[Lemma~2.5]{mattila} for more details. Then the conclusion follows by compactness of the unit Euclidean sphere.
\end{proof}

The next theorem, usually known as Besicovitch's covering theorem for the Euclidean space in the literature, comes as a consequence of Lemma~\ref{lem:WBCP-Rn}.

\begin{theorem} \label{thm:strong-bcp-Rn} Let $\R^n$ be equipped with the Euclidean distance. There are integers $N\geq 1$ and $Q\geq 1$ with the following properties. Let $A$ be a bounded subset of $\R^n$ and $\bb$ be a family of balls such that each point of $A$ is the center of some ball of $\bb$.
\begin{itemize}
\item[$(i)$] (Besicovitch covering property) There is a countable subfamily $\f\subset \bb$ such that the balls in $\f$ cover $A$ and every point in $\R^n$ belongs to at most $N$ balls in $\f$. 

\smallskip

\item[$(ii)$] (strong Besicovitch covering property) There are countable subfamilies $\bb_1,\dots,\bb_Q \subset \bb$ covering $A$ such that each $\bb_i$ is disjointed, namely,
$$A \subset \bigcup_{i=1}^Q \bigcup_{B\in \bb_i} B$$ and $B\cap B' = \emptyset \;\;\text{for } B, B' \in \bb_i \; \text{with } B\not= B'.$
\end{itemize}
\end{theorem}

For a proof of Theorem~\ref{thm:strong-bcp-Rn} we refer for instance to~\cite[Theorem 2.7]{mattila}. See also Section~\ref{sect:finite-dim-metrics} for the connections between Lemma~\ref{lem:WBCP-Rn} and Theorem~\ref{thm:strong-bcp-Rn}~$(i)$ in the general metric setting. 

\medskip

We explain now how the fact that every Radon measure over $\R^n$ is of Vitali type with respect to the Euclidean distance can be obtained as a rather simple consequence of the strong Besicovitch covering property given by Theorem~\ref{thm:strong-bcp-Rn}~$(ii)$.

\begin{theorem}  \label{thm:Vitali-Radon-meas-Rn}
Radon measures over $\R^n$  are of Vitali type with respect to the Euclidean distance. 
\end{theorem}

As a consequence of Theorem~\ref{thm:Vitali-Radon-meas-Rn} together with the results of Section~\ref{sect:Vitali-type-meas}, we get the following corollary.

\begin{corollary}
The conclusions of Theorem~\ref{thm:diff-vitali-meas} and Corollary~\ref{cor:diff-thm-from-vitali-type} hold true for every Radon measure over the Euclidean space.
\end{corollary}

\begin{proof} [Proof of Theorem~\ref{thm:Vitali-Radon-meas-Rn}]
Let $\lambda$ be a Radon measure over $\R^n$. Let $A\subset \R^n$ and let $\bb$ be a family of Euclidean balls such that each point of $A$ is the center of arbitrarily small balls of $\bb$. We may assume that $\lambda(A)>0$. Let us assume first that $A$ is bounded and hence has finite $\lambda$ measure. By outer regularity of $\lambda$, there is an open set $U$ such that $A\subset U$ and 
$$\lambda(U) \leq (1 + (4Q)^{-1})\, \lambda(A)$$
where $Q$ is given by Theorem~\ref{thm:strong-bcp-Rn}. By Theorem~\ref{thm:strong-bcp-Rn}~$(ii)$ we can find countable subfamilies $\bb_1,\dots,\bb_Q \subset \bb$ such that each $\bb_i$ is disjointed and 
$$A\subset \bigcup_{i=1}^Q \bigcup_{B\in \bb_i} B \subset U.$$
Then $$\lambda(A) \leq \sum_{i=1}^Q \lambda \left( \bigcup_{B\in \bb_i} B \right)$$ hence there is an $i$ such that $$\lambda(A) \leq Q\, \lambda \left( \bigcup_{B\in \bb_i} B \right).$$ Further we can find finitely many disjointed balls $B_1,\dots,B_{k_1} \in \bb_i$ such that $$\lambda(A) \leq 2Q\, \lambda \left( \bigcup_{j=1}^{k_1} B_j \right).$$ Letting $$A_1 := A \setminus \bigcup_{j=1}^{k_1} B_j,$$ we get $$\lambda(A_1) \leq \lambda\left(U \setminus \bigcup_{j=1}^{k_1} B_j \right) = \lambda(U) - \lambda\left(\bigcup_{j=1}^{k_1} B_j \right) \leq u \lambda(A)$$
with $u:= 1-(4Q)^{-1} <1$. 
Now $A_1$ is contained in the open set $\R^n \setminus \cup_{j=1}^{k_1} B_j$ and therefore we can find an open set $U_1$ such that $A_1 \subset U_1 \subset \R^n \setminus \cup_{j=1}^{k_1} B_j$ and $$\lambda(U_1) \leq (1 + (4Q)^{-1})\, \lambda(A_1).$$ As before there are finitely many disjointed balls $B_{k_1+1},\dots,B_{k_2}$ for which $B_j \subset U_1$ and $$\lambda(A_2) \leq  u \lambda(A_1) \leq u^2 \lambda(A)$$
where $$A_2 := A_1 \setminus \bigcup_{j=k_1+1}^{k_2} B_j = A \setminus \bigcup_{j=1}^{k_2} B_j.$$
Clearly, the balls $B_1,\dots,B_{k_2}$ are disjointed. After $m$ steps, we get $$\lambda\left(A \setminus \bigcup_{j=1}^{k_m} B_j \right) \leq u^m \lambda(A)$$ 
with the balls $B_1,\dots,B_{k_m}\in \bb$ that are disjointed and the fact that $\lambda$ is of Vitali type with respect to the Euclidean distance follows since $u<1$.

\medskip

To get rid of the assumption that $A$ is bounded, we write $\R^n$ as the union of a countable collection of closed cubes $\overline{Q_i}$ such that the corresponding open cubes $Q_i$ are disjointed and such that $\lambda(\R^n \setminus \cup_{i\geq 1} Q_i) = 0$. This is possible since $\lambda(V)$ can be positive for at most countably many parallel hyperplanes $V$. Applying the above arguments to the sets $A\cap Q_i$ and noting that $\lambda(A \setminus \cup_{i \geq 1} Q_i) = 0$, we get that $\lambda$ is of Vitali type with respect to the Euclidean distance. 
\end{proof}

\begin{remark} The conclusions of Lemma~\ref{lem:WBCP-Rn} and Theorem~\ref{thm:strong-bcp-Rn} hold for families of open balls. However, Theorem~\ref{thm:Vitali-Radon-meas-Rn} does not hold in its full generality when considering family of open balls in the definition of Vitali type measures, see~\cite[Example~2.20]{AFP}.
\end{remark}

\begin{remark} \label{rk:directionally-limited-metrics} More general metric spaces satisfy the strong Besicovitch covering property, see Theorem~\ref{thm:strong-bcp-Rn}~$(ii)$ for the statement in the Euclidean case. We refer for instance to~\cite[2.8.9]{federer} for the notion of directionally limited metrics and to~\cite[2.8.14]{federer} for an extension of Theorem~\ref{thm:strong-bcp-Rn}~$(ii)$ for such metrics. Finite dimensional normed vector spaces fit in particular into this more general setting. We shall however not go further in that direction in these notes. In the next sections, we shall instead investigate in more details the geometric content of the Besicovitch and weak Besicovitch covering properties and their connections with the validity of the differentiation theorem for every locally finite Borel regular measure in the metric setting.
\end{remark}

\section{$\sigma$-finite dimensional metrics} \label{sect:finite-dim-metrics}

We present in this section the notion of $\sigma$-finite dimentional metrics introduced by Preiss in~\cite{P}. This notion  characterizes complete separable metric spaces on which the differentiation theorem holds for every locally finite Borel regular measure, see Theorem~\ref{thm:preiss-diff}. 

\medskip

We start with a model case to explain without to much technicalities how one can get from such kind of notions the validity of the differentiation theorem for every locally finite Borel regular measure.

\medskip

\begin{definition} [Besicovitch covering property] \label{def:BCP}
We say that a distance $d$ on a space $X$ satisfies the {\em Besicovitch covering property} (BCP) if there exists an integer $N\geq 1$ such that the following holds. Let $A$ be a bounded subset of $X$ and $\mathcal{B}$ be a family of balls such that each point of $A$ is the center of some ball of $\mathcal{B}$. Then there is a countable subfamily $\mathcal{F}\subset \mathcal{B}$ such that the balls in $\mathcal{F}$ cover $A$ and every point in $X$ belongs to at most $N$ balls in $\mathcal{F}$, namely, 
\begin{equation*}
\chi_A \leq \sum_{B \in \mathcal{F}} \chi_B \leq N
\end{equation*}
where $\chi_A$ denotes the characteristic function of the set $A$.
\end{definition}

\medskip

With this terminology, Theorem~\ref{thm:strong-bcp-Rn}~$(i)$ exactly says that the Euclidean distance satisfies BCP. We show in the next theorem that the validity of BCP is sufficient to imply the validity of the differentiation theorem for every locally finite Borel regular measure over a separable metric space. Recall that, given a metric space $(X,d)$, we denote by $M$ the class of all locally finite Borel regular measures over $(X,d)$.

\begin{theorem} \label{thm:diff-thm-from-bcp}
Let $(X,d)$ be a separable metric space. Assume that $d$ satisfies BCP. Then the differentiation theorem holds for every $\lambda \in M$.
\end{theorem}

\begin{proof}
Every $\lambda \in M$ satisfies the assumption of Theorem~\ref{thm:diff-thm-from-max-op} below with $C=1$ and $N$ given by the validity of BCP for $d$.
\end{proof}

\begin{theorem} \label{thm:diff-thm-from-max-op}
Let $(X,d)$ be a separable metric space and let $\lambda \in M$. Assume there exist constants $C>0$ and  $N>0$ such that the following holds. For every bounded set $A\subset X$ and every family $\bb$ of balls in  $(X,d)$ such that each point of $A$ is the center of some ball of $\bb$, there is a countable subfamily $\f \subset \bb$ such that 
\begin{equation*} \label{e:diff-thm-from-max-op}
 \lambda (A) \leq C \lambda \left(\bigcup_{B\in \f} B\right) \;\text{ and }\; \sum_{B\in \f} \chi_B \leq N.
\end{equation*}
Then the differentiation theorem holds for $\lambda$.
\end{theorem}

\begin{proof}
The proof given here is based on a weak type $(1,1)$ inequality for the maximal operator. We refer for instance to~\cite{deGuzman} or~\cite{stein} for more specialized results in that direction. First, it is sufficient to prove~\eqref{e:lebesgue-diff} for $f\in L^1(\lambda)$ multiplying our original function by the characteristic function of an open set where $f$ is $\lambda$-integrable and then exhausting $X$ by a countable union of such open sets. 

\medskip

Next, for $f\in L^1(\lambda)$ and $x\in X$, set $$Mf(x) := \sup_{r>0} \frac{1}{\lambda(B(x,r))} \int_{B(x,r)} |f(y)| \,d\lambda y.$$
Let $A$ be a bounded subset of $X$. For $\alpha >0$, set $$A_\alpha := \{x\in A : Mf(x) > \alpha\}.$$
For each $x\in A_\alpha$, we can find a ball $B(x,r_x)$ such that $$\lambda(B(x,r_x)) < \frac{1}{\alpha}\int_{B(x,r_x)} |f(y)| \,d\lambda y.$$ By assumption, one can extract from this family of balls a sequence $B_1,B_2,\dots$ such that 
\begin{equation*}
\begin{split}
\lambda(A_\alpha) &\leq  C \lambda  \left(\bigcup_{j\geq 1} B_j\right) \leq C \sum_{j\geq 1} \lambda(B_j)\\ &\leq \frac{C}{\alpha} \sum_{j\geq 1} \int_{B_j} |f(y)| \,d\lambda y = \frac{C}{\alpha}  \int_X \sum_{j\geq 1} \chi_{B_j}(y) |f(y)| \,d\lambda y \leq \frac{CN}{\alpha} \,\|f\|_1.
\end{split}
\end{equation*}
Since this holds for every bounded set $A$, we get the following weak type (1,1) inequality for the maximal operator,
\begin{equation} \label{e:weak-1,1-max-op}
\lambda (\{x\in X : Mf(x) > \alpha\}) \leq \frac{CN}{\alpha} \,\|f\|_1.
\end{equation}

\medskip

For $g\in L^1(\lambda)$ and $r>0$, denote by $g_r$ the function defined by $$g_r(x): = \frac{1}{\lambda(B(x,r))}\int_{B(x,r)} g(y) \,d\lambda y$$
and set $$\Omega g(x): = \limsup_{r\rightarrow 0} |g_r(x) - g(x)|.$$ 
We have $\Omega g(x) \leq Mg(x) + g(x)$. Therefore, using~\eqref{e:weak-1,1-max-op} and Chebychev's inequality, we get 
\begin{equation*}
\begin{split}
\lambda (\{x\in X& : \Omega g(x)>\alpha\}) \\
&\leq \lambda \left(\left\{x\in X : Mg(x) > \frac{\alpha}{2} \right\}\right) + \lambda \left(\left\{x\in X : |g(x)| > \frac{\alpha}{2} \right\}\right) \\
&\leq \frac{C'}{\alpha} \,\|g\|_1
\end{split}
\end{equation*}
for every $\alpha >0$ and for some constant $C'>0$.

\medskip

To conclude the proof of~\eqref{e:lebesgue-diff}, let $f\in L^1(\lambda)$. Let $h\in L^1(\lambda)$ and assume that $h$ is continuous. Then $\Omega h$ is identically zero on the support of $\lambda$ and  hence $\lambda$-a.e. Next, we have $\Omega f \leq \Omega (f-h) + \Omega h = \Omega (f-h)$, therefore
$$\lambda (\{x\in X : \Omega f (x) >\alpha\}) \leq \lambda (\{x\in X : \Omega (f-h) (x) >\alpha\}) \leq \frac{C'}{\alpha} \,\|f-h\|_1.$$
By density of continuous functions in $L^1(\lambda)$, we get 
$$\lambda (\{x\in X : \Omega f (x) >\alpha\}) = 0$$ for every $\alpha >0$. Hence $\Omega f(x) =0$ for $\lambda$-a.e.~$x\in X$ which proves~\eqref{e:lebesgue-diff}.
\end{proof}

\medskip

\begin{remark} Note that every doubling measure satisfies the assumption of Theorem~\ref{thm:diff-thm-from-max-op}. This can be deduced from Theorem~\ref{thm:basic-5r-covering} and is left to the reader as an easy exercise. Hence one can also recover from Theorem~\ref{thm:diff-thm-from-max-op} the validity of the differentiation theorem for doubling measures over separable metric spaces.
\end{remark}

\medskip

More generally, results about derivatives of measures similar to those presented in Section~\ref{sect:Vitali-type-meas} can be obtained with the assumption of Theorem~\ref{thm:diff-thm-from-max-op}. Namely,

\begin{theorem}
Let $(X,d)$ be a separable metric space. Let $\lambda \in M$ and assume that the assumption of Theorem~\ref{thm:diff-thm-from-max-op} holds. Then for every $\mu \in M$, we have
\begin{equation} \label{e:diff-abs-cont-part}
\mu_\lambda(A) = \int_A D(\mu,\lambda,x)\,d\lambda x \;\; \text{for every } \lambda \text{-measurable set } A,
\end{equation}
where $\mu_\lambda$ denotes the absolute continuous part of the Lebesgue decomposition of $\mu$ relative to $\lambda$.
\end{theorem}

In particular, if $d$ satisfies BCP, then~\eqref{e:diff-abs-cont-part} holds for every $\lambda, \mu \in M$.

\begin{proof}
The fact that there exists $f\in L^1_{loc}(\lambda)$ such that $$ \mu_\lambda(A) = \int_A f(x)\, d\lambda x$$ for every $\lambda$-measurable set $A$ follows from Radon-Nikodym theorem for $\sigma$-finite measures. Theorem~\ref{thm:diff-thm-from-max-op} implies the validity of the differentiation theorem for $\lambda$ and in particular $f(x) = D(\mu_\lambda ,\lambda,x)$ for a.e.~$x\in X$. Then~\eqref{e:diff-abs-cont-part} follows from Lemma~\ref{lem:density-singular-meas-from-weak-type-1,1} below. We apply it to $\nu=\mu_s$, where $\mu_s$ denotes the singular part of the Lebesgue decomposition of $\mu$ relative to $\lambda$, to get that $D(\mu_s,\lambda,x) = 0$, and hence $f(x)= D(\mu_\lambda,\lambda,x)=D(\mu,\lambda,x)$, for $\lambda$-a.e.~$x\in X$. 
\end{proof}

\begin{lemma} \label{lem:density-singular-meas-from-weak-type-1,1}
Let $(X,d)$ be a separable metric space. Let $\lambda \in M$ and assume that the assumption of Theorem~\ref{thm:diff-thm-from-max-op} holds. Let $\nu\in M$ be such that $\nu \perp \lambda$. Then $D(\nu,\lambda,x) = 0$ for $\lambda$-a.e.~$x\in X$.
\end{lemma} 

\begin{proof} The proof is very similar to the proof of Theorem~\ref{thm:diff-thm-from-max-op}. Let $A$ be a Borel set such that $\nu(A) = 0 = \lambda (X\setminus A)$. Let $\varepsilon >0$ and let $U$ be an open set such that $A\subset U$ and $\nu(U) \leq \varepsilon$. Let $V$ be a bounded subset of $A$. For $\alpha >0$, set 
$$V_\alpha := \{ x\in V: \overline{D}(\nu,\lambda,x) >\alpha\}.$$
For each $x\in V_\alpha$, we can find a ball $B(x,r_x) \subset U$ such that $$\lambda(B(x,r_x)) < \frac{1}{\alpha} \, \nu(B(x,r_x)).$$ By assumption we can extract from this family of balls of countable subfamily $\f$ such that
\begin{equation*} 
 \lambda (V_\alpha) \leq C \lambda \left(\bigcup_{B\in \f} B\right) \;\text{ and }\; \sum_{B\in \f} \chi_B \leq N.
\end{equation*}
Arguing in a similar way as in the proof of Theorem~\ref{thm:diff-thm-from-max-op}, it follows that 
\begin{equation*}
\lambda (V_\alpha) \leq \frac{CN}{\alpha} \,\nu(U) \leq\frac{CN}{\alpha} \, \varepsilon .
\end{equation*}
Since this holds for every $\varepsilon >0$, we get $\lambda (V_\alpha) = 0$ for every $\alpha >0$. Therefore $$\lambda ( \{ x\in V: \overline{D}(\nu,\lambda,x) >0\}) = 0$$ and the conclusion follows since the latter holds for every bounded subset of $A$ and since $\lambda (X\setminus A) = 0$.
\end{proof}

\medskip

We introduce now a slightly weaker geometric condition than BCP. We say that a family $\bb$ of balls in a metric space $(X,d)$ is a \textit{Besicovitch family of balls} if, first, for every ball $B\in \bb$ with center $x_B$, we have $x_B\not \in B'$ for all $B'\in\bb$, $B\not= B'$, and, second, $\bigcap_{B\in \mathcal{B}} B \not= \emptyset$. 

\begin{definition} [Weak Besicovitch covering property] \label{def:WBCP}
We say that a distance $d$ on a space $X$ satisfies the \textit{weak Besicovitch covering property} (WBCP) if there is an integer $K\geq 1$ such that $\card \bb \leq K$ for every Besicovitch family $\bb$ of balls in $(X,d)$.
\end{definition}

\medskip
Going back to the Euclidean setting, Lemma~\ref{lem:WBCP-Rn} exactly shows that the Euclidean distance satisfies WBCP.

\medskip

It can easily be seen that BCP implies WBCP. But, as the terminology suggests, WBCP is in general strictly weaker than BCP, see~\cite[Example 3.4]{LeDonneRigot3}, and the comment after Proposition~\ref{prop:doubling-bcp-wbcp} for the more precise relationship between BCP and WBCP.

\medskip

There is however an important class of metric spaces, namely, doubling metric spaces, where BCP and WBCP are equivalent. Recall that a metric space is said to be doubling if there is an integer $D\geq 1$ such that for each $0<r<\infty$, each ball in $(X,d)$ of radius $2r$ can be covered by a family of at most $D$ balls of radius $r$. 

\begin{proposition} \label{prop:doubling-bcp-wbcp}
Let $(X,d)$ be a doubling metric space. Then $d$ satisfies BCP if and only if $d$ satisfies WBCP.
\end{proposition}

For a proof of Proposition~\ref{prop:doubling-bcp-wbcp}, we refer for instance to~\cite[Proposition 3.7]{LeDonneRigot3} which follows actually closely the arguments of the proof of~\cite[Theorem 2.7]{mattila} where the validity of BCP for the Euclidean distance, namely, Theorem~\ref{thm:strong-bcp-Rn}~$(i)$, is deduced from the validity of WBCP.

\medskip

Although a distance $d$ that satisfies WBCP may not satisfy BCP in general, it satisfies a weak form of BCP that can be stated as follows. There is an integer $N\geq 1$ such that the following holds. Let $A$ be a bounded subset of $X$. Let $\mathcal{B}$ be a family of balls such that each point of $A$ is the center of some ball of $\mathcal{B}$ and such that either $\sup\{r:\; B(x,r)\in \mathcal{B}\}=+\infty$ or $B(x,r)\in \mathcal{B} \mapsto r$ attains only an isolated set of values in $(0,+\infty)$. Then there is a countable subfamily $\mathcal{F}\subset \mathcal{B}$ such that the balls in $\mathcal{F}$ cover $A$ and every point in $X$ belongs to at most $N$ balls in $\mathcal{F}$. Mild modifications of the arguments presented in this section can be used to prove the validity of the differentiation theorem for every locally finite Borel regular measure from this weak form of Besicovitch covering property. 

\medskip

We conclude this section with the notion of $\sigma$-finite dimensional metric that generalizes WBCP and is due to Preiss, see~\cite{P}. This notion involves the decomposition of the ambient space into countably many pieces on which an ad hoc version of WBCP holds.

\begin{definition} [$\sigma$-finite dimensional metric]
Let $(X,d)$ be a metric space. We say that $d$ is finite dimensional on a subset $Y\subset X$ if there exist constants $K\geq 1$ and $0<r\leq \infty$ such that $\card \mathcal{B} \leq K$ for every Besicovitch family $\bb$ of balls in $(X,d)$ centered on $Y$ with radius $<r$. We say that $d$ is \textit{$\sigma$-finite dimensional} if $X$ can be written as a countable union of subsets on which $d$ is finite dimensional. 
\end{definition}

With this terminolgy, the fact that WBCP holds on a metric space $(X,d)$ exactly means that the distance $d$ is finite dimensional on $X$ for some constant $K\geq 1$ and with $r=\infty$ in the previous definition. Examples of $\sigma$-finite dimensional metrics are Riemannian metrics over Riemannian manifolds of class $\geq 2$, see~\cite[Chapter~2]{federer}. On the contrary, (non-Riemannian) sub-Riemannian distances are not $\sigma$-finite dimensional, see~\cite[Theorem~7.5]{LeDonneRigot3}.
We end this section with the following characterization whose proof can be found in~\cite{P}.

\begin{theorem} [{\cite[Theorem~1]{P}}]\label{thm:preiss-diff}
Let $(X,d)$ be a complete separable metric space. The  differentiation theorem holds for every locally finite Borel regular measure over $(X,d)$ if and only if $d$ is $\sigma$-finite dimensional.
\end{theorem}

\section{Weak Besicovitch covering property} \label{sect:wbcp}

We devote this last section to a closer study of the weak Besicovitch covering property, see Definition~\ref{def:WBCP}. Besides the fact that it implies the validity of the differentiation theorem for every locally finite Borel regular measures over a separable metric space, WBCP is a useful tool that can be used to deduce global properties of a metric space from local ones, and can for instance be used to study arbitrary measures.

\medskip

We begin with a couple of general facts. We first observe that WBCP is a property for family of balls in a metric space $(X,d)$ that may not be preserved under (even arbitrarily small) perturbations of the metric. 

\begin{theorem}[{\cite[Theorem 1.6]{LeDonneRigot2}}] \label{thm:destroybcp}
Let $d$ be a distance on a space $X$. Assume that there exists an accumulation point in $(X,d)$. Then, for every $0<c<1$, there exists a distance $d_c$ on $X$ that does not satisfy WBCP and such that $cd \leq d_c \leq d$.
\end{theorem}

See also~\cite[Theorem~3]{P} for a version of this result about $\sigma$-finite dimensional metrics. Conversely, there are examples of metric spaces $(X,d)$ that do not satisfy WBCP but for which there exists for every $\epsilon >0$, a distance $d_\epsilon$ on $X$ that satisfies WBCP and is such that $d\leq d_\epsilon \leq (1+\epsilon) d$. See the example of the stratified first Heisenberg group, Example~\ref{ex:stratified-heis} below. WBCP is in particular far from being preserved by a bi-Lipschitz change of metric. 

\medskip

On the other hand, let $d$ and $\rho$ be two distances on a space $X$ and assume that every ball with respect to $\rho$ is a ball with respect to $d$, with the same center but possibly a different radius. Then the validity of WBCP for $d$ implies the validity of WBCP for $\rho$. For instance if WBCP holds for a distance $d$, then, for every $0<s<1$, the snowflake distance $d^s$ satisfies WBCP as well. Note that on the contrary it is well known that a metric space $(X,d)$ and its snowflakes $(X,d^s)$, $0<s<1$, have for many other purposes significantly different behaviour.

\medskip

Our next observation concerns product of metric spaces. Given two metric spaces $(X,d_X)$ and $(Y,d_Y)$, there are many ways to define distances on their product. If $(X,d_X)$ and $(Y,d_Y)$ both satisfy WBCP, then WBCP may fail for classical choices of distances on $X\times Y$ such as their $l^p$-mean for $1\leq p < \infty$. Such examples are the following ones. Given $1\leq p < \infty$ and $s>p$, let $d$ be the distance on $\R \times \R$ defined by $d((x,y),(x',y')):=(|x'-x|^p + |y'-y|^{p/s})^{1/p}$. It is the $l^p$-mean of two distances on $\R$ that satisfy WBCP, but WBCP does not hold on $(\R\times\R,d)$, see~\cite[Lemma~3.2]{LeDonneRigot1}. However, considering the max distance on a product of metric spaces preserves the validity of WBCP. Namely,

\begin{theorem} \label{thm:WBCP-product-spaces} \cite[Theorem 3.16]{LeDonneRigot3} Let $(X,d_X)$ and $(Y,d_Y)$ be two metric spaces. Assume that $d_X$ and $d_Y$ satisfy WBCP on $X$ and $Y$ respectively. Then the max distance $$d_{X\times Y}((x,y),(x',y')):=  \max(d_X(x,x'),d_Y(y,y'))$$ satisfies WBCP on $X\times Y$.
\end{theorem}

\medskip

The proof of Theorem~\ref{thm:WBCP-product-spaces} relies on tools from graph theory, namely, Ramsey's Theorem. We refer to~\cite{LeDonneRigot3} for a proof. We also refer to the latter paper, and in particular to its Section~3, for further general results about WBCP.

\medskip

Until recently, there were only few known examples of metric spaces satisfying WBCP. As far as we know, finite dimensional normed vector spaces, see Remark~\ref{rk:directionally-limited-metrics}, were the main known such examples. We end this section with recent results that enlarge this picture.

\medskip

The setting we are considering in the rest of these notes is the one of homogeneous groups equipped with homogeneous distances. Roughly speaking, a homogeneous group is a Lie group equipped with an appropriate family of dilations. A homogeneous distance on a homogeneous group is a left-invariant distance that is one-homogeneous with respect to the family of dilations, see below for more detailed definitions. To some extend, homogeneous groups equipped with homogeneous distances generalize naturally finite dimensional normed vector spaces. Due to the presence of translations and dilations, they provide for instance a setting where many aspects of classical analysis and geometry can be carried out. However, it is also well known that, for many purposes, their behaviour can be significantly different from the behaviour of finite dimensional normed vector spaces. The discussion below about the validity or non validity of WBCP will give some more evidence about these differences. We also recall that, beyond such a priori considerations, homogeneous groups equipped with homogeneous distances form an important framework because of their occurrences in many settings, see for instance~\cite{folland-stein}, or the introduction in~\cite{LeDonneRigot3} and the references therein.

\medskip

Before stating the results, we introduce some definitions and terminology. We refer to~\cite{LeDonneRigot3} for a complete presentation, see also Examples~\ref{ex:step-1},~\ref{ex:stratified-heis} and~\ref{ex:non-standard-heis} below for some explicit examples.

\medskip

A positive grading of a finite dimensional real Lie algebra $\g$ is a family $(V_t)_{t\in (0,+\infty)}$ of vector subspaces of $\g$ where all but finitely many of the $V_t$'s are $\{0\}$ and such that 
$$\g=  \oplus_{t>0} V_t \;\;  \text{  with } [V_s, V_t]\subset V_{s+t} \; \text{ for all } s,t >0.$$
Here
$
[V,W]: = \Span\{[X,Y] :\; X\in V,\ Y\in W\} .
$
Given a positive grading $(V_t)_{t>0}$ of a Lie algebra and $t>0$, the subspace $V_t$ is called the degree $t$ layer of the grading. We say that a Lie algebra is positively graduable if it admits a positive grading. In general, a positively graduable Lie algebra admits several positive gradings that are not isomorphic as graded Lie algebras. We say that a Lie algebra is graded if it is positively graduable and endowed  with a fixed positive grading. 

\medskip

For a graded Lie algebra $\g= \oplus_{t>0} V_t $ and $r>0$, the associated dilation of factor $r$ is defined as the unique linear map $\delta_r:\g\to\g$ such that $\delta_r (X) = r^t X$ for $X\in V_t$. The family of associated dilations $(\delta_r)_{r>0}$ is a one-parameter group of Lie algebra automorphisms.

\medskip

We say that a connected and simply connected Lie group $G$ is graded if its Lie algebra is graded. For a graded group $G$ and $r>0$, we define, with a slight abuse of notation and terminology, the associated dilation $\delta_r : G \rightarrow G$ of factor $r$ as the unique Lie group automorphism such that $\delta_r\circ\exp = \exp\circ\,\delta_r$ where $\exp: \g \to G$ denotes the exponential map. Recall that graded groups are connected and simply connected nilpotent Lie groups hence the exponential map is a diffeomorphism from $\g$ to $G$. 

\medskip

We say that a distance $d$ on a graded group $G$ with associated family of dilations $(\delta_r)_{r>0}$ is homogeneous if it is left-invariant, that is, $d(p\cdot q,p\cdot q') = d(q,q')$ for all $p$, $q$, $q'\in G$, and one-homogeneous with respect to the associated family of dilations $(\delta_r)_{r>0}$, that is, $d(\delta_r(p),\delta_r(q)) = r d(p,q)$ for all $p$, $q\in G$ and all $r>0$.

\medskip

Homogeneous distances do exist on a graded group if and only if, for all $t<1$, degree $t$ layers of the associated positive grading are $\{0\}$. We call such groups \textit{homogeneous}. For general graded groups, one can consider homogeneous quasi-distances. For simplicity, we restrict ourselves in these notes to homogeneous groups and we refer to~\cite{LeDonneRigot3} for the more general case of graded groups. 

\medskip
Note that we do not require topological assumptions in the definition of homogeneous distances. It can indeed be proved that homogeneous distances on a homogeneous group induce the manifold topology, see~\cite[Proposition 2.26]{LeDonneRigot3}. Note also that a homogeneous group equipped with a homogeneous distance is doubling, hence the validity of WBPC is equivalent to the validity of BCP in this setting. Furthermore, it can be proved that a homogeneous distance on a homogeneous group is $\sigma$-finite dimensional if and only if it satisfies WBCP, see~\cite[Proposition~6.1]{LeDonneRigot3}.

\medskip

The following theorem gives a characterization of homogeneous groups that admit homogeneous distances for which WBCP holds. We say that a graded group $G$ with associated positive grading of its Lie algebra given by $\g = \oplus_{t>0} V_t$ has \textit{commuting different layers} if $[V_t,V_s] = \{0\}$ for all $t$, $s>0$ such that $t\not= s$. 

\begin{theorem} \cite[Corollary~1.3]{LeDonneRigot3} \label{thm:bcp-homogeneous-groups}
Let $G$ be a homogeneous group. There exist homogeneous distances on $G$ for which WBCP holds if and only if $G$ has commuting different layers.
\end{theorem} 

Note that, given a homogeneous group with commuting different layers, a layer of its associated positive grading may not commute with itself, and there are indeed examples of homogeneous groups with commuting different layers beyond the Abelian case, see below for few of them.

\medskip

A large class of homogeneous groups is given by stratified groups, also known as Carnot groups in the literature (although the latter terminology might sometimes more specifically be used when such groups are equipped with sub-Riemannian or sub-Finsler distances). We recall that a stratification of step $s$ of a finite dimensional real Lie algebra $\g$ is a positive grading of the form 
$$\g = V_1 \oplus V_2 \oplus \cdots \oplus V_s \;\;  \text{ where }[V_1,V_j] = V_{j+1}\; \text{ for all } 1\leq j\leq s $$ for some integer $s\geq 1$ and where $V_s \not= \{0\}$ and $V_{s+1} = \{0\}$. Equivalently, a stratification is a positive grading whose degree one layer generates $\g$ as a Lie algebra. We say that a Lie algebra is stratified of step $s$ if it is graded with associated positive grading that is a stratification of step $s$. We say that a connected and simply connected Lie group is stratified of step $s$ if its Lie algebra is stratified of step $s$. It can easily be seen that a stratified group has commuting different layers if and only if it is of step $1$ or $2$. Hence Theorem~\ref{thm:bcp-homogeneous-groups} has the following corollary.

\begin{corollary} \cite[Corollary~1.4]{LeDonneRigot3} \label{thm:bcp-stratified-groups}
Let $G$ be a stratified group. There exist homogeneous distances on $G$ for which WBCP holds if and only if $G$ is of step 1 or 2.
\end{corollary}

Let us now give some examples together with an outline of the main ideas involved in the proof of Theorem~\ref{thm:bcp-homogeneous-groups}.

\begin{example} \label{ex:step-1} Stratified groups of step $1$ equipped with homogeneous distances can be identified with finite dimensional normed vector spaces. As already explained, in such a case the validity of WBCP was known for a long time,  see Remark~\ref{rk:directionally-limited-metrics}.
\end{example}

The simplest example of non Abelian positively graduable Lie algebra is given by the first Heisenberg Lie algebra.

\medskip

The first Heisenberg Lie algebra $\h$ is the 3-dimensional Lie algebra that admits a basis $(X,Y,Z)$ where the only non trivial bracket relation is $[X,Y] = Z$. With no loss of generality, let us fix such a basis  of $\h$.

\medskip

The first Heisenberg group $\HH$ is the connected and simply connected Lie group whose Lie algebra is $\h$. Using exponential coordinates of the first kind, we write $p \in \HH$ as $p=\exp(x X + y Y + z Z)$ and we identify $p$ with $(x,y,z)\in \R^3$. Using the Baker-Campbell-Hausdorff formula, the group law is given by 
\begin{equation*}\label{e:law-heisenberg}
(x,y,z)\cdot (x,y,z) 
= (x + x', y+y', z+z'+\frac{1}{2} (x y' - y x'))~.
\end{equation*}

\begin{example} [Stratified first Heisenberg group] \label{ex:stratified-heis}
The first Heisenberg Lie algebra $\h$ is stratifiable of step 2. Namely,
\begin{equation*} 
\h = V_1 \oplus V_2 \quad \text{where }\; V_1:=\Span\{X,Y\},\;  V_2:=\Span{Z},
\end{equation*}
is a stratification that we call the \textit{standard stratification}. Note that stratifications are unique up to Lie algebra automorphisms of graded Lie algebras. Hence there is no loss of generality here to work with the standard stratification of $\h$ rather than with other possible stratifications.

\medskip

We call \textit{stratified first Heisenberg group} the first Heisenberg group viewed as a stratified group whose Lie algebra is equipped with the standard stratification. Associated dilations on the stratified first Heisenberg group are given by
\begin{equation*} 
\delta_r(x,y,z):= (r x,r y,r^2 z).
\end{equation*}

\medskip

There are several classical examples of homogeneous distances on the stratified first Heisenberg group, such as the Kor\'anyi (see~\eqref{e:def-koranyi-dist}) or sub-Riemannian distances. The following class of examples is due to Hebisch et Sikora. For $\gamma>0$, set $$A_\gamma := \{(x,y,z)\in \HH: x^2+y^2+z^2 \leq \gamma^2\}$$
and, for $p, q \in \HH$,
\begin{equation} \label{e:defdistdalpha}
d_\gamma(p,q): = \inf\{r>0: \delta_{1/r}(p^{-1}\cdot q) \in A_\gamma \}.
\end{equation}

\medskip

It is proved in~\cite{Hebisch_Sikora} that there is $\gamma^*\geq 2$, such that, for every $0<\gamma\leq \gamma^*$, $d_\gamma$ defines a homogeneous distance on the stratified first Heisenberg group. Note that the unit ball centered at the origin for such a distance is given by the set $A_\gamma$, that is, a Euclidean ball centered at the origin with a small enough radius. We refer to~\cite{Hebisch_Sikora} for a complete statement about existence of such homogeneous distances on arbitrary homogeneous groups.

\medskip

Going back to the validity of WBCP, the main result in~\cite{LeDonneRigot2} is the validity of WBCP for the homogeneous distances $d_\gamma$ on the stratified first Heisenberg group.

\begin{theorem} \cite[Theorem~1.14]{LeDonneRigot2} \label{thm:bcp-heis} Let $\gamma >0$ be such that $d_\gamma$ (see~\eqref{e:defdistdalpha}) is a homogeneous distance on the stratified first Heisenberg group. Then WBCP holds on $(\HH,d_\gamma)$.
\end{theorem}

When $\gamma = 2$, $d_2$ defines a homogeneous distance on the stratified first Heisenberg group that is related to the Kor\'anyi distance $d_K$ via the following formula $$d_2(0,(x,y,z)) = (2\sqrt{2})^{-1} \, ((x^2 + y^2) + d_K(0,(x,y,z))^2)^{1/2},$$ 
where the Kor\'anyi distance from the origin is given by 
\begin{equation} \label{e:def-koranyi-dist}
d_K(0,(x,y,z)) := ((x^2 + y^2)^2 + 16 z^2)^{1/4}.
\end{equation}

\medskip

A proof that, for every $\epsilon>0$, the homogeneous distance $d_\epsilon$ given by $$d_\epsilon (0,(x,y,z)): =   (\epsilon(x^2 + y^2) + d_K(0,(x,y,z))^2)^{1/2}$$ satisfies WBCP on the stratified first Heisenberg group will be given in a forthcoming paper~\cite{Nicolussi-Rigot}. This will in particular give a sequence of homegeneous distances that are as close as one wants to the Kor\'anyi distance and that satisfy WBCP. However, it was noticed independently and at the same time in~\cite{KoranyiReimann} and~\cite{SawyerWheeden} that the Kor\'anyi distance itself does not satisfy WBCP.

\medskip

Theorem~\ref{thm:bcp-heis} has been extended to stratified free nilptotent groups of step 2 in~\cite[Theorem~4.5]{LeDonneRigot3}. The existence of homogeneous distances for which WBCP holds on finite dimensional normed vector spaces and on stratified free nilptotent groups of step 2 is the first crucial geometric step in the proof of the "if" part in Theorem~\ref{thm:bcp-homogeneous-groups}. This implication can indeed be deduced from these two model cases using the algebraic structure of graded Lie algebras with commuting different layers and ad hoc submetries, together with some of the general results stated at the beginning of this section. We refer to~\cite[Section~2]{LeDonneRigot1} for the definition of submetries and their relationship with the validity of WBCP, and to~\cite[Section~4]{LeDonneRigot3} for a complete proof of the "if" part in Theorem~\ref{thm:bcp-homogeneous-groups}.
\end{example}

\begin{example} [Non-standard first Heisenberg groups] \label{ex:non-standard-heis}
The first Heisenberg Lie algebra $\h$ admits positive gradings that are not stratifications. Namely, for $\alpha >1$, we call \textit{non standard grading of exponent $\alpha$} the positive grading of $\h$ given by
\begin{equation*} 
\h =W_1 \oplus W_\alpha \oplus W_{\alpha+1}
\end{equation*}
where $ W_1:=\Span\{X\}$, $W_\alpha:=\Span\{Y\}$, $W_{\alpha+1}=\Span\{Z\}$. Dilations associated to the non standard grading of exponent $\alpha$ are given by 
\begin{equation*} 
\delta_r(x,y,z) := (r x, r^\alpha y,r^{\alpha+1} z)~.
\end{equation*}

\medskip

We call \textit{non standard first Heisenberg group of exponent $\alpha$} the first Heisenberg group viewed as a graded group whose Lie algebra is endowed with the non standard grading of exponent $\alpha$. We have $[W_1,W_\alpha] = W_{\alpha+1}$ hence non standard Heisenberg groups do not have commuting different layers. 

\medskip

The non existence of continuous homogeneous quasi-distances for which \linebreak WBCP holds on non standard first Heisenberg groups, see~\cite[Theorem~5.6]{LeDonneRigot3}, is the first crucial geometric step in the proof of the "only if" part in Theorem~\ref{thm:bcp-homogeneous-groups}. The general case of homogeneous groups that do not have commuting different layers can indeed be obtained from this model case using an algebraic relationship between positive grading of Lie algebras that do not have commuting different layers and non standard gradings of the first Heisenberg Lie algebra, together with ad hoc submetries. We refer to~\cite[Section~5]{LeDonneRigot3} for a complete proof of the "only if" part in Theorem~\ref{thm:bcp-homogeneous-groups}. 
\end{example}

\end{document}